\documentclass[10pt,twoside]{amsart}

\usepackage{color}
\usepackage{amsthm}
\usepackage{amssymb}
\usepackage{amsmath} 
\usepackage{mathrsfs}
\usepackage{graphicx}
\usepackage{bm}  
\usepackage{dsfont} 
\usepackage[nice]{nicefrac}   
\usepackage{verbatim}
\usepackage{enumerate}
\usepackage{url}

\newtheorem{thm}{Theorem}
\newtheorem{cor}[thm]{Corollary}
\newtheorem{lem}[thm]{Lemma}

\theoremstyle{definition}
\newtheorem{defn}{Definition}
\newtheorem{ex}{Example}
\newtheorem{rem}{Remark}
\def    \C      {\mathds{C}}
\def    \Z      {\mathds{Z}}
\def    \der     {{\scriptscriptstyle  D}}
\def    \sst     {\scriptscriptstyle}
\def    \inv    {{\scriptscriptstyle  \langle-1\rangle}}
\def    \K      {{\frak K}}
\def    \bo     {\boldsymbol}
\def    \mbf    {\mathbf}
\def    \ofx    {(\mathbf{x})}
\def    \punt   {{\boldsymbol{.}}}
\def    \A      {\mathcal{A}}
\def    \B      {\mathcal{B}}
\def    \ab     {{\frak a}}
\def    \deg     {\mathrm{deg}}

\allowdisplaybreaks

\title{Outcomes of the Abel identity}
\author{Pasquale Petrullo}
\begin{document}
\maketitle
\begin{center}\small
Dipartimento di Matematica e Informatica\\
Universit\`a degli Studi della Basilicata\\
Via dell'Ateneo Lucano 10, 85100 Potenza, Italia.\\
\texttt{p.petrullo@gmail.com}
\end{center}
\begin{abstract}
Through symbolic methods, we state explicit formulae for
Tchebychev, Gegenbauer, Meixner, Mittlag-Leffler, and Pidduck
polynomials. This is done by underlining the crucial role played
by the Abel identity in revisiting the Lagrange inversion formula and the theory of the Riordan
arrays.\\
\end{abstract}
{\
}
\section{Introduction}
Given a variable $x$ and a complex number $a$, then the sequence $(a_n(x,a))_{n\geq 0}$ of Abel polynomials is defined by $a_n(x,a)=x(x+na)^{n-1}$. They hold some interesting combinatorial insight~\cite{Sag}, and they satisfy nice identities such as
\begin{align}
\label{id:AIcl}(x+y)^n&=\sum_{k=0}^n\binom{n}{k}a_k(x,a)(y-ka)^{n-k},\\
\label{id:BIcl}a_n(x+y)&=\sum_{k=0}^n\binom{n}{k}a_k(x,a)a_{n-k}(y,a),\\
\label{id:Dabcl}D_xa_n(x,a)&=na_{n-1}(x+a,a).
\end{align}
In particular, in \cite{Com} the Abel identity \eqref{id:AIcl} and the binomial identity \eqref{id:BIcl} are obtained in a simple way by means of the key-property \eqref{id:Dabcl}. The polynomial sequences $(p_n(x))_{n\geq 0}$, with $\deg p_n=n$, satisfying $p_n(x+y)=\sum_{k=0}\binom{n}{k}p_k(y)p_{n-k}(y)$ are often called polynomials of binomial type, or binomial sequences. So, because of \eqref{id:BIcl} we see that $(a_n(x,a))_{n\geq 0}$ is of binomial type. Some investigation on the recursive properties of the matrix of the coefficients of a binomial sequence was carried out by Knuth~\cite{Kn}. The wider class of Sheffer polynomials, that includes binomial sequences, have been deeply studied by Roman~\cite{Rom}, through the {\lq\lq modern\rq\rq} umbral calculus originated with the seminal paper of Rota~\cite{Rota}. In 1994, Rota and Taylor~\cite{RT} presented a renewed version of the umbral calculus which is based on the notion of {\lq\lq umbra\rq\rq}. They called it the {\lq\lq classical umbral calculus\rq\rq}. Essentially, this symbolic method consists of an infinite set $\A=\{\alpha,\beta,\gamma,\ldots,\}$ of variables, called umbrae, and of the linear functional evaluation $E\colon\C[\A]\to\C$ holding the uncorrelation property $E[\alpha^n\beta^m\cdots\gamma^l]=E[\alpha^n]E[\beta^m]\cdots E[\gamma^l]$. Roughly speaking, we pass from the umbral calculus to the classical umbral calculus by associating each linear functional $L\colon\C[x]\to\C$ with the umbra $\alpha$ satisfying $L[x^n]=E[\alpha^n]$. The alphabet $\A$ is extended by adding copies $\alpha',\alpha'',\ldots$ of each umbra $\alpha$, and by introducing new symbols called auxiliary umbrae. More precisely, if $k$ is a positive integer then $k\punt\alpha$ is the auxiliary umbra defined to satisfy $E[(k\punt\alpha)^n]=E[(\alpha'+\alpha''+\cdots+\alpha^{\sst(k)})^n]$, for all $n\geq 0$, and the umbral Abel polynomial $\ab_n(\gamma,\alpha)=\gamma(\gamma+n\punt\alpha)^{n-1}$ may be considered. By setting $\gamma=y$, $\alpha=a$ and by assuming $E[y^n]=y^n$ for all $n\geq 0$, we get $E[\ab_n(y,a)]=a_n(y,a)$, or more compactly $\ab_n(y,a)\simeq a_n(y,a)$. By means of umbral Abel polynomials, Rota, Shen and Taylor~\cite{RST} have obtained a very nice characterization of the polynomials of binomial type: these polynomial sequences $(p_n(x))_{n\geq 0}$ precisely are the sequences such that $p_n(x)\simeq \ab_n(x,\alpha)$ for all $n\geq 0$, for a suitable $\alpha$. Hence, from an umbral point of view all polynomial sequences of binomial type have the form of an (umbral) Abel polynomial. In 2001, Taylor~\cite{Tay} showed that Abel polynomials encode not only binomial sequences, but the whole class of Sheffer sequences: $(s_n(x))_{n\geq 0}$ is a Sheffer sequence if and only if $s_n(x)\simeq\ab_n(x+\gamma,\alpha)$. In the same year Di Nardo and Senato~\cite{DNS1} began their exploration of the classical umbral calculus, by enriching the syntax and by stating for the first time the umbral Abel identity. Then, with Niederhausen~\cite{DNNS} they have also given a more complete description of the Sheffer sequences, by underlining connections with the Lagrange Inversion Formula, and by showing the way to move towards the Riordan arrays theory~\cite{BaHe,DFR,SGWW,Spr1,Spr2}. A deeper study of the Riordan arrays by means of Rota's umbrae can be found in \cite{AMPT}. In the last ten years, a classical umbral calculus approach to the classical cumulant theory have been carried out~\cite{DNS2,RS}. Abel polynomials come back on the scene in a paper by Di Nardo, Petrullo and Senato~\cite{DNPS}, where they play a crucial role in defining a unifying framework for classical, boolean and free cumulants, and provide a bridge with the combinatorics of parking functions through the volume polynomial of Pitman and Stanley~\cite{PitSta}. See also~\cite{PetSen} on this subject. Finally, a wider range of identities encoded by the umbral Abel polynomials can be found in \cite{Pet}.

In this paper, we present an elementary proof of the umbral Abel identity, by showing that \eqref{id:Dabcl} generalizes to umbral Abel polynomials. From this starting point, we show how some fundamental facts relating Sheffer sequences, Lagrange Inversion Formula and Riordan arrays arise in a trivial way and with one-line proofs. Moreover, we give an application of this methods in determining explicit formulae for the following class of polynomials: Tchebychev (II kind), Gegenbauer, Meixner (I kind), Mittlag-Laffler and Pidduck.
\section{The Abel identity and the Lagrange Inversion Formula.}
Let $\mbf{x}\cup\{z\}$ be a finite set of commuting variables.
Denote by $\C[\mbf{x}]$ the ring of polynomials with complex
coefficients in the variables of $\mbf{x}$, and by
$\C[\mbf{x}][[z]]$ the ring of all formal power series (f.p.s) of
type
\[f(z)=\sum_{n\geq 0}p_n\ofx\frac{z^n}{n!},\]
with $p_n\ofx\in\C[\mbf{x}]$ for all $n\geq 0$. As is customary,
the coefficient of $z^n$ in a f.p.s. $f(z)$ is denoted by
$[z^n]f(z)$, with $[z^0]f(z)=f(0)$. For all
$f(z)\in\C[\mbf{x}][[z]]$ such that $f(0)=1$, we introduce a
symbol $\alpha$, called the {\it umbra} of $f(z)$. The set $\A$ of
all umbrae will be called the {\it base alphabet}. The umbra of
$1$ is called the {\it augmentation} and it is denoted by
$\varepsilon$, the umbra of $1+z$ is named the {\it singleton} and
it is denoted by $\chi$, the umbra $\beta$ of $e^{e^z-1}$ is
called the {\it Bell umbra}. Then, the {\it evaluation} is the
only $\C[\mbf{x}]$-linear functional
$E\colon\C[\mbf{x}][\A]\to\C[\mbf{x}]$ such that
\begin{enumerate}
\item $E[{\alpha}^n]=n![z^n]f(z)$, if and only if $\alpha$ is the umbra of $f(z)$,
\item $E[{\alpha}^{n}{\gamma}^{m}\cdots {\delta}^{l}]=E[{\alpha}^{n}]E[{\gamma}^{m}]\cdots E[{\delta}^{l}]$, whenever $\alpha,\gamma,\ldots,\delta$ are pairwise distinct umbrae ({\it uncorrelation property}).
\end{enumerate}
Here, $\C[\mbf{x}][\A]$ denotes the ring of polynomials in the
variables of $\A$ and with coefficients in $\C[\mbf{x}]$. To avoid
any confusion, complex numbers in $\C$ will be denoted by
$a,b,\ldots$, polynomials in $\C[\mbf{x}]$ by $p\ofx,
q\ofx,\ldots$, umbrae in $\A$ by Greek letters
$\alpha,\beta,\ldots$, and elements in $\C[\mbf{x}][\A]$ by
$\bo{p},\bo{q},\ldots$. Each  $\bo{p}\in\C[\mbf{x}][\A]$ will be
named an {\it umbral polynomial}. We say that $p_n\ofx$ is the
$n$th {\it moment} of $\bo{p}$ if $E[\bo{p}^n]=p_n\ofx$, and that
the sequence $(p_n\ofx)_{n\geq 0}$ is {\it represented} by
$\bo{p}$ if $E[\bo{p}^n]=p_n\ofx$ for all $n\geq 0$. Note that, by
construction each sequence such that $p_0\ofx=1$ is represented by
a unique umbra $\alpha\in\A$. Now, we extend the action of $E$ to
the f.p.s. of type $e^{\bo{p} z}$, so that
\[E[e^{\bo{p} z}]=E\left[\sum_{n\geq 0}\bo{p}^n\frac{z^n}{n!}\right]=1+\sum_{n\geq 1}E[\bo{p}^n]\frac{z^n}{n!}.\]
We set $f_{\bo{p}}(z)=E[e^{\bo{p} z}]$ and say that
$f_{\bo{p}}(z)$ is {\it generating function} of $\bo{p}$. Note
that $e^{(\alpha+\gamma)z}=e^{\alpha z}e^{\gamma z}$ and because
of the uncorrelation property we have
$f_{\alpha+\gamma}(z)=f_\alpha(z)f_\gamma(z)$. Nevertheless, if
$n\in\Z$ then $e^{n\alpha z}=(e^{\alpha z})^n$ but
$f_{n\alpha}(z)\neq f_\alpha(z)^n$. Analogously, $e^{\gamma\alpha
z}=e^{\gamma\left(\log e^{\alpha z}\right)}$ but
$f_{\gamma\alpha}(z)\neq f_\gamma\left(\log\,f_\alpha(z)\right)$.
Then, in order to preserve some notational coherence, we define new
symbols, denoted by $n\punt\alpha$ and $\gamma\punt\alpha$,
called the {\it auxiliary umbrae}. A new alphabet $\B$ is defined to
contain them, and the action of $E$ is then extended to
$\C[\mbf{x}][\A\cup\B]$. In detail, for all $\bo{p}$ and $\bo{q}$
the auxiliary umbra $\bo{p}\punt\bo{q}$ is defined to satisfy
$E[e^{(\bo{p}\punt \bo{q})z}]=
f_{\bo{p}}\left(\log\,f_{\bo{q}}(z)\right)$, so that $E[e^{(a\punt
\alpha)z}]= f_\alpha(z)^a$ and $E[e^{(\alpha\punt a)z}]=
f_\alpha(az)$. Iterations of type $\delta\punt \gamma\punt\alpha$
are allowed, and if $\beta$ is the Bell umbra then we obtain
$f_{\gamma\punt\beta\punt\alpha}(z)=f_\gamma\left(\log\,f_\beta\left(\log\,f_\alpha(z)\right)\right)=f_\gamma\left(f_\alpha(z)-1\right)$.
Sometimes, $\gamma\punt\beta\punt\alpha$ is named the {\it
composition umbra} of $\gamma$ with $\alpha$. Further auxiliary
umbrae are denoted by $\alpha_\der$, with $\alpha\in\A$.  They are
named {\it derivative umbrae} and are defined by
$E[{\alpha_\der}^n]=E[D_\alpha\alpha^n]$, for all $n\geq 1$, where
$D_\alpha\colon \alpha^n\mapsto n\alpha^{n-1}$. Hence, $\alpha_\der$ satisfy
$f_{\alpha_\der}(z)-1=zf_\alpha(z)$. From now on, we will write
$\bo{p}\simeq \bo{q}$ ($\bo{p}$ is {\it equivalent} to $\bo{q}$)
whenever $E[\bo{p}]=E[\bo{q}]$, and $\bo{p}\equiv \bo{q}$
($\bo{p}$ is {\it similar} to $\bo{q}$) if and only if
$\bo{p}^n\simeq \bo{q}^n$ for all $n\geq 0$. Also, we will write
$e^{\bo{p}z}\simeq f(z)$ instead of $E[e^{\bo{p}z}]=f(z)$, and
$e^{\bo{p}z}\simeq e^{\bo{q}z}$ if
$E[e^{\bo{p}z}]=E[e^{\bo{q}z}]$. This provides compact expressions
like $\alpha\punt a\equiv a\alpha$ and $b\punt a\equiv ba$. In particular, we notice that
the restriction of the dot-operation on $\C$ (and $\C[\mbf{x}]$)
can be thought as the usual multiplication on $\C$. Also, we get
the duality relation $\chi\punt\beta\equiv\beta\punt\chi\equiv 1$.
As it is shown by equivalence \eqref{id:comp}, the moments of dot-operations of type
$\gamma\punt\beta\punt\alpha_\der$ holds a binomial-like expansion
in terms of the moments of $\alpha$ and $\gamma$. In fact, since
$e^{(\gamma\punt\beta\punt\alpha_\der)z}\simeq
e^{\gamma\left(zf_\alpha(z)\right)}$ then we obtain
\[e^{(\gamma\punt\beta\punt\alpha_\der)z}\simeq \sum_{i\geq 0}\gamma^i\frac{z^i}{i!}e^{(i\punt\alpha)z}\simeq \sum_{i\geq 0}\sum_{j\geq 0}\gamma^i(i\punt\alpha)^j\frac{z^i}{i!}\frac{z^j}{j!}.\]
Finally, we gain
\begin{equation}\label{id:comp}
(\gamma\punt\beta\punt\alpha_\der)^n\simeq \sum_{k=0}^{n}\binom{n}{k}\gamma^k(k\punt\alpha)^{n-k}.
\end{equation}
We stress that, for any fixed $\alpha$ there are pairwise distinct
auxiliary umbrae of type $1\punt\alpha$, $1\punt
1\punt\alpha$,\ldots, and of type $\alpha\punt 1$,
$\alpha\punt1\punt 1$, \ldots, representing the same sequence of
moments (i.e. similar auxiliary umbrae). We will refer to any of
them by means of more compact notations such as
$\alpha',\alpha'',\ldots$. It follows that
$n\punt\alpha\equiv\alpha'+\alpha''+\cdots+\alpha^{\sst(n)}$, and
then $n\punt\alpha$ behaves as a sum of $n$ distinct similar
umbrae. Moreover, to our aim it is not important to distinguish
umbrae in $\A$ from the auxiliary umbrae in $\B$. So that, the
term {\lq\lq umbra\rq\rq} will refer to any unspecified symbol in
$\A\cup\B$. The pair $(\C[\mbf{x}][\A\cup\B],E)$ provides a
so-called {\it saturated umbral calculus} over $\C[\mbf{x}]$. See \cite{DNS1,RT} for more details. From
now on, all sums of type $\alpha+(-k\punt\gamma)$ will be
abbreviated by $\alpha-k\punt\gamma$. Hence, we have
$k\punt\alpha-k\punt\alpha\equiv-k\punt\alpha+k\punt\alpha\equiv\varepsilon
\equiv 0$, for all $k$.
\begin{defn}
Let $\alpha,\gamma$ be umbrae, and let $n$ be a positive integer. Then, the umbral polynomial
\begin{equation}\label{def:ab}
\ab_n(\gamma,\alpha)=\gamma(\gamma+n\punt\alpha)^{n-1},
\end{equation}
is called an {\it umbral Abel polynomial}, or simply an {\it Abel polynomial}.
\end{defn}
If $\alpha$ and $\gamma$ are replaced by a complex number $a$ and
a variable $x\in\mbf{x}$ respectively, then we gain
$\ab_n(x,a)\simeq x(x+na)^{n-1}$ (being $n\punt a\equiv na$).
Hence, the definition above returns the classical Abel polynomials
$a_n(x,a)=x(x+na)^{n-1}$. The well-known property $D_x
a_n(x,a)=na_n(x+a,a)$ generalizes to $\ab_n(\gamma,\alpha)$ as
follows.
\begin{lem}
For all umbrae $\alpha,\gamma$, and for all $n\geq 1$ we have
\begin{equation}\label{id:Dab}
D_\gamma\ab_n(\gamma,\alpha)\simeq n\ab_{n-1}(\gamma+\alpha',\alpha).
\end{equation}
\end{lem}
\begin{proof}
Elementary computations show that
\begin{equation}\label{id:1}
D_\gamma\,\gamma(\gamma+n\punt\alpha)^{n-1}=n\gamma\left(\gamma+n\punt\alpha\right)^{n-2}+\left(n\punt\alpha\right)\left(\gamma+n\punt\alpha\right)^{n-2}.
\end{equation}
It is also clear that
\begin{equation}\label{id:2}
n\,\gamma\left(\gamma+n\punt\alpha\right)^{n-2}\simeq n\gamma\left(\gamma+\alpha+(n-1)\punt\alpha\right)^{n-2}.
\end{equation}
Moreover, if $\alpha',\alpha'',\ldots,\alpha^{\sst(n)}$ are distinct umbrae similar to $\alpha$, then we have
\begin{align}\label{id:3}
\nonumber\left(n\punt\alpha\right)\left(\gamma+n\punt\alpha\right)^{n-2}&\simeq (\alpha'+\alpha''+\cdots+\alpha^{\sst(n)})\left(\gamma+\alpha'+\alpha''+\cdots+\alpha^{\sst(n)}\right)^{n-2}\\
\nonumber&\simeq \sum_{i=1}^n\alpha^{\sst(i)}\left(\gamma+\alpha^{\sst(i)}+(n-1)\punt\alpha\right)^{n-2}\\
&\simeq n\,\alpha\left(\gamma+\alpha+(n-1)\punt\alpha\right)^{n-2}.
\end{align}
The equivalence \eqref{id:Dab} follows by comparing \eqref{id:1},
\eqref{id:2} and \eqref{id:3}.
\end{proof}
\begin{thm}[Umbral Abel identity]
For all umbrae $\alpha,\gamma,\delta$ we have
\begin{equation}\label{id:AI}
(\delta+\gamma)^n\simeq \sum_{k=0}^n\binom{n}{k}(\delta+k\punt\alpha)^{n-k}\,\gamma(\gamma-k\punt\alpha)^{k-1}.
\end{equation}
\end{thm}
\begin{proof}
By thinking of $((\delta+\gamma)^n)_{n\geq 0}$ and
$(\ab_n(\gamma,-1\punt\alpha))_{n\geq 0}$ as basis of a suitable
ring of umbral polynomials in the distinguished variable $\gamma$,
we may refer to the unique sequence $(b_n)_{n\geq 0}$ of
coefficients such that
\begin{equation}\label{id:4}
(\delta+\gamma)^n=\sum_{k=0}^n b_k\,\ab_k(\gamma,\alpha).
\end{equation}
Let us apply ${D_\gamma}^i$ to both sides in \eqref{id:AI}. Via \eqref{id:Dab} we gain
\[(n)_i(\delta+\gamma)^{n-i}\simeq\sum_{k=0}^nb_k\,(k)_i\ab_{k-i}(\gamma-i\punt\alpha',-1\punt\alpha),\]
with $(n)_i=n(n-1)\cdots(n-i+1)$. Finally, because of
$-i\punt\alpha+i\punt\alpha\equiv 0$, by evaluating on
$\gamma=k\punt\alpha'$ both sides in \eqref{id:4} we have
$\ab_{i-k}(0,\alpha)\simeq \delta_{i,k}$, and then
$(n)_k(\delta+k\punt\alpha)^{n-k}\simeq k!\,b_k$.
\end{proof}
As a by product we obtain the following generalization.
\begin{cor}
For all $\alpha, \gamma, \delta$, and for all polynomials $q(\gamma)$ of degree $n$ we have
\begin{equation}\label{id:GAI}
q(\delta+\gamma)\simeq \sum_{k=0}^n\frac{{D_\gamma}^kq(\delta+k\punt\alpha)}{k!}\,\gamma(\gamma-k\punt\alpha)^{k-1},
\end{equation}
with ${D_\gamma}^kq(\delta+k\punt\alpha)$ denoting the $k$th derivative of $q(\gamma)$ evaluated on $\gamma=\delta+k\punt\alpha$.
\end{cor}
\begin{proof}
It follows from \eqref{id:AI} because of the linearity of $E$ and of the ${D_\gamma}^k$'s.
\end{proof}
If $E[\alpha]\neq 0$ then $f_\alpha(z)-1$ admits compositional
inverse $\left(f_\alpha(z)-1\right)^\inv$. Then we denote by
$\alpha^\inv$ an umbra such that
$f_{\alpha^\inv}(z)-1=\left(f_\alpha(z)-1\right)^\inv$, or
equivalently such that
$\alpha\punt\beta\punt\alpha^\inv\equiv\alpha^\inv\punt\beta\punt\alpha\equiv\chi$,
where $\chi$ is the singleton. For all umbrae $\gamma$ and
$\alpha$, we define $\K_{\gamma,\alpha}$ to be an auxiliary umbra
with moments
\begin{equation}\label{def:K}
{\K_{\gamma,\alpha}}^n\simeq\gamma(\gamma-n\punt\alpha)^{n-1}, \text{ for all $n\geq 1$}.
\end{equation}
Moreover, we will abbreviate $\K_{\alpha,\alpha}$ by $\K_\alpha$. The umbra $\K_\alpha$ plays a fundamental role within umbral cumulant theory~\cite{DNPS}. Note that, the equivalence \eqref{id:Dab} may be expressed as
\begin{equation}\label{id:DabK}
\K_{\gamma_\der,\alpha}\equiv(\K_{\gamma-1\punt\alpha,\alpha})_\der.
\end{equation}
\begin{thm}[Lagrange Inversion Formula]
For all umbrae $\alpha$ and $\gamma$ we have
\begin{equation}\label{id:LIF2}
(\gamma\punt\beta\punt{\alpha_\der}^\inv)^n\simeq \gamma(\gamma-n\punt\alpha)^{n-1}.
\end{equation}
\end{thm}
\begin{proof}
By means of \eqref{id:comp}, \eqref{def:K}, and \eqref{id:AI} with $\delta=0$, we have $\K_{\sst\gamma,\alpha}\punt\beta\punt\alpha_\der\equiv\gamma$, and finally $\K_{\gamma,\alpha}\equiv\gamma\punt\beta\punt{\alpha_\der}^\inv$.
\end{proof}
In terms of generating functions, the equivalence \eqref{id:LIF2}
is nothing but one of the most general versions of the famous
Lagrange Inversion Formula,
\[[z^n]f_\gamma\left(\left(zf_\alpha(z)\right)^\inv\right)=\frac{1}{n}[z^{n-1}]f_\gamma'(z)\left(\frac{1}{f_\alpha(z)}\right)^{n}.\]
We stress that, in this setting, it is merely a corollary of the
Abel identity stated by a one-line proof. Finally, since
$\chi\equiv\varepsilon_\der$, via \eqref{id:DabK} and
\eqref{id:LIF2} we obtain
${\alpha_\der}^\inv\equiv\K_{\varepsilon_\der,\alpha}\equiv(\K_{-1\punt\alpha,\alpha})_\der$.
However, it is $\K_{-1\punt\alpha,\alpha}\equiv-1\punt\K_\alpha$,
and then we obtain
\begin{equation}\label{id:Kinv}
\alpha_\der\equiv{(-1\punt\K_\alpha)_\der}^\inv.
\end{equation}
\section{Explicit formulae for polynomial sequences}
Let $x\in\mbf{x}$. A polynomial sequence $(s_n(x))_{n\geq 0}$ is
said to be a {\it Sheffer sequence} if and only if there are
$A(z),B(z)\in\C[[z]]$ with $A(0)\neq 0\neq B(0)$, and such that
\begin{equation}\label{def:sheff}
\sum_{n\geq 0}s_n(x)\frac{z^n}{n!}=A(z)e^{xzB(z)}.
\end{equation}
When $A(z)=1$ the polynomial sequence is said to be a {\it
binomial sequence}, or an {\it associated sequence}. Sheffer
sequences can be also characterized as the sequences
$(s_n(x))_{n\geq 0}$, with $\deg\,s_n(x)=n$, satisfying the
following {\it Sheffer identity},
\begin{equation}\label{id:sheff}
s_n(x+y)=\sum_{k=0}^n\binom{n}{k}p_k(x)s_{n-k}(y),
\end{equation}
with $(p_n(x))_{n\geq 0}$ being a binomial sequence. Analogously,
the binomial sequences are the unique sequences $(p_n(x))_{n\geq
0}$, with $deg\,p_n(x)=n$, satisfying the {\it binomial identity},
\begin{equation}\label{id:bin}
p_n(x+y)=\sum_{k=0}^n\binom{n}{k}p_k(x)p_{n-k}(y).
\end{equation}
Abel polynomials $(a_n(x,a))_{n\geq 0}$ are binomial with
$B(z)=\left(ze^{-az}\right)^\inv$. By iterating \eqref{id:Dab} we get
$D_\gamma^k\ab_n(\gamma,\alpha)\simeq
(n)_k\ab_{n-k}(\gamma+k\punt\alpha)$. Then, by replacing $\alpha$
with $-1\punt\alpha$ in \eqref{id:GAI}, and choosing
$p(\gamma)=\ab_n(\gamma,\alpha)$, we recover the binomial identity
of umbral Abel polynomials,
\[\ab_n(\gamma+\delta,\alpha)\simeq \sum_{k=0}^n\binom{n}{k}\ab_k(\gamma,\alpha)\ab_n(\delta,\alpha').\]
In the following, by writing a {\lq\lq Sheffer sequence\rq\rq} or
a {\lq\lq binomial sequence\rq\rq} we will always mean a {\lq\lq
Sheffer sequence with $A(0)=B(0)=1$\rq\rq} or a {\lq\lq binomial
sequence with $B(0)=1$\rq\rq}. It is not difficult to see that we
are restricting our analysis to the case $s_n(x)=x^n+\text{lower
terms}$. Let $\gamma$ and $\alpha$ be umbrae such that $e^{\gamma
z}\simeq A(z)$ and $e^{\alpha z}\simeq B(z)$, then the umbral
translation of \eqref{def:sheff} is
\begin{equation}\label{id:sheffumbra}
s_n(x)\simeq (\gamma+x\punt\beta\punt\alpha_\der)^n,
\end{equation}
and Sheffer sequences are moments of umbrae of type
$\gamma+x\punt\beta\punt\alpha_\der$. Thanks to \eqref{id:comp},
straightforward computations give
\begin{equation}\label{id:sheffmom}
s_n(x)\simeq \sum_{k=0}^n\binom{n}{k}(\gamma+k\punt\alpha)^{n-k}x^k,
\end{equation}
and the coefficient $s_{n,k}$ of $x^k$ in $s_n(x)$ have the following closed expression,
\begin{equation}\label{id:sheffcoeff}
s_{n,k}\simeq \binom{n}{k}(\gamma+k\punt\alpha)^{n-k}.
\end{equation}
From now on, the polynomial sequence $(s_n(x))_{n\geq 0}$ defined
by \eqref{id:sheffumbra} will be named the {\it Sheffer sequence}
of $(\gamma,\alpha)$. Moreover, we will write
\[(\gamma,\alpha)=(s_n(x))_{n\geq 0}\]
to express that one of the equivalent identities
\eqref{id:sheffumbra}, \eqref{id:sheffmom}, and
\eqref{id:sheffcoeff} hold. The {\it Appell sequences} are the
Sheffer sequences of $(\gamma,\varepsilon)$, the {\it binomial
sequences}, also named {\it associated sequences}~\cite{Rom}, are the Sheffer
sequences of $(\varepsilon,\alpha)$. If we lift to an umbral
level, then we may appreciate that Sheffer sequences have the form
of Abel polynomials.
\begin{thm}[Abel representation of Sheffer sequences]
A polynomial sequence $(s_n(x))_{n\geq 0}$ is the Sheffer sequence of $(\gamma,\alpha)$ if and only if
\[s_n(x)\simeq (x+\K_{\gamma,\alpha})\left(x+\K_{\gamma,\alpha}+n\punt\K_{\alpha}\right)^{n-1}.\]
\end{thm}
\begin{proof}
We have
$\gamma+x\punt\beta\punt\alpha_\der\equiv\K_{\gamma,\alpha}\punt\beta\punt\alpha_\der+x\punt\beta\punt\alpha_\der\equiv(x+\K_{\gamma,\alpha})\punt\beta\punt{(-1\punt\K_{\alpha})_\der}^\inv$,
and the claim follows via \eqref{id:LIF2}.
\end{proof}
\begin{cor}[Appell sequences]
A sequence $(s_n(x))_{n\geq 0}$ is the Appell sequence of $(\gamma,\varepsilon)$ if and only if
\[s_n(x)\simeq \left(x+\gamma\right)^{n}.\]
\end{cor}
\begin{proof}
Set $\alpha=\varepsilon$ in the theorem above.
\end{proof}
\begin{cor}[Binomial sequences]
A sequence $(s_n(x))_{n\geq 0}$ is the binomial sequence of $(\varepsilon,\alpha)$ if and only if
\[s_n(x)\simeq x\left(x+n\punt\K_\alpha\right)^{n-1}.\]
\end{cor}
\begin{proof}
Set $\gamma=\varepsilon$ in the theorem above.
\end{proof}
Given the Sheffer sequences $(\gamma,\alpha)=(s_{n}(x))_{n\geq 0}$
and $(\eta,\delta)=(r_n(x))_{n\geq 0}$, one may consider the
so-called {\lq\lq umbral composition\rq\rq}
\[(\gamma,\alpha)(\eta,\delta)=(s_n(\bo{r}(x)))_{n\geq 0},\]
which is the polynomial sequence defined by~\cite{Rom}
\[s_n(\bo{r}(x))=\sum_{k=0}^n s_{n,k}r_k(x).\]
Being $s_n(x)\simeq (\gamma+x\punt\beta\punt\alpha_\der)^n$ and $r_n(x)\simeq (\eta+x\punt\beta\punt\delta_\der)^n$, then the umbra representing $(s_n(\bo{r}(x))_{n\geq 0}$ is obtained by replacing $x$ with $\eta+x\punt\beta\punt\delta_\der$ in $\gamma+x\punt\beta\punt\delta_\der$. Hence, we get
\[s_n(\bo{r}(x))\simeq \left(\gamma+\eta\punt\beta\punt\alpha_\der+x\punt\beta\punt\delta_\der\punt\beta\punt\alpha_\der\right)^n.\]
Via generating functions it is quite simple to see that $\delta_\der\punt\beta\punt\alpha_\der\equiv\left(\alpha+\delta\punt\beta\punt\alpha_\der\right)_\der$, and the following theorem is proved.
\begin{thm}[Umbral composition of Sheffer sequences]
The umbral composition of the Sheffer sequence of $(\gamma,\alpha)$
with the Sheffer sequence of $(\eta,\delta)$ is the Sheffer sequence of $(\gamma+\delta\punt\beta\punt\alpha_\der,\alpha+\delta\punt\beta\punt\alpha_\der)$.
In symbols, we have
\begin{equation}\label{id:umbcomp}
(\gamma,\alpha)(\eta,\delta)=(\gamma+\eta\punt\beta\punt\alpha_\der,\alpha+\delta\punt\beta\punt\alpha_\der).
\end{equation}
\end{thm}
Note that, $(\varepsilon,\varepsilon)=(x^n)_{n\geq 0}$ is the
identity with respect to the umbral composition. Let
$A(z),B(z)\in\C[[z]]$ with $A(0)\neq 0\neq  B(0)$. The {\it
exponential Riordan array} defined by $(A(z),B(z))$ is the
infinite lower triangular matrix $(s_{n,k})_{n,k}$ determined by
\begin{equation}\label{def:Rio}
s_{n,k}=n![z^{n}]A(z)\frac{\left(zB(z)\right)^{k}}{k!}.
\end{equation}
As we have done for Sheffer sequences, we restrict ourselves to
the case $A(0)=B(0)=1$, so that there are $\alpha$ and $\gamma$
such that $f_\gamma(z)=A(z)$ and $f_\alpha(z)=B(z)$. This way, we
obtain
\begin{equation}\label{id:Rio}
s_{n,k}\simeq \binom{n}{k}(\gamma+k\punt\alpha)^{n-k}.
\end{equation}
The Riordan array \eqref{id:Rio} will be named the {\it Riordan
array} of $(\gamma,\alpha)$. The following relation among Sheffer
sequences and Riordan arrays comes trivially.
\begin{thm}
The infinite lower triangular matrix $(s_{n,k})_{n,k}$ is the
Riordan array of $(\gamma,\alpha)$ if and only if the polynomial
sequence $(s_n(x))_{n\geq 0}$ defined by
\[s_n(x)=\sum_{k=0}^ns_{n,k}x^k\]
is the Sheffer sequence of $(\gamma,\alpha)$.
\end{thm}
\begin{proof}
It follows by comparing \eqref{id:sheffcoeff} and \eqref{id:Rio}.
\end{proof}
The Fundamental Theorem of the Riordan arrays~\cite{WW} can be easily stated
by replacing $x$ with a umbra $\eta$ in \eqref{id:sheff} and
\eqref{id:sheffmom}.
\begin{thm}[Fundamental Theorem of the Riordan arrays]
Let $(s_{n,k})_{n,k}$ be the Riordan array of $(\gamma,\alpha)$ and
let $(a_n)_{n\geq 0}$ be the sequence represented by $\eta$. Then
we have
\[(s_{n,k})_{n,k}\left(
                   \begin{array}{c}
                     1 \\
                     a_1 \\
                     a_2 \\
                     \vdots \\
                   \end{array}
                 \right)=
                 \left(
                   \begin{array}{c}
                     1 \\
                     b_1 \\
                     b_2 \\
                     \vdots \\
                   \end{array}
                 \right)
\]
if and only if $(b_n)_{n\geq 0}$ is represented by $\gamma+\eta\punt\beta\punt\alpha_\der$.
\end{thm}
Let
\[s_n(x)=\sum_{k=0}^ns_{n,k}x^k \text{ and }r_n(x)=\sum_{k=0}^nr_{n,k}x^k.\]
It is plain that
\[s_n(\bo{r}(x))=\sum_{k=0}^n\left(\sum_{i=k}^ns_{n,i}r_{i,k}\right)x^k.\]
However, $\sum s_{n,i}r_{i,k}$ is nothing but the
$(n,k)$-entry in the array $(s_{n,k})_{n,k}(r_{n,k})_{n,k}$. This
way, the identity \eqref{id:umbcomp} also encodes the
multiplication of Riordan arrays: the product of the Riordan array of $(\gamma,\alpha)$ with the Riordan array of $(\eta,\delta)$ is the
Riordan array of $(\gamma+\delta\punt\beta\punt\alpha_\der,\alpha+\delta\punt\beta\punt\alpha_\der)$.
Clearly, the identity is given by $(\varepsilon,\varepsilon)$.
Moreover, being $s_{n,n}=1$ for all $n\geq 0$, we see that
$(s_{n,k})_{n,k}$ is invertible and the set of all the Riordan
arrays  of type \eqref{id:Rio} is a group with respect to the
usual matrix multiplication. Thanks to \eqref{id:umbcomp}, we may
write
\[(\gamma,\alpha)^{-1}=(-1\punt\K_{\gamma,\alpha},-1\punt\K_\alpha),\]
where $(\gamma,\alpha)^{-1}$ denotes the unique array such that
\[(\gamma,\alpha)^{-1}(\gamma,\alpha)=(\gamma,\alpha)(\gamma,\alpha)^{-1}=(\varepsilon,\varepsilon).\]
This means that the Riordan array of $(-1\punt\K_{\gamma,\alpha},-1\punt\K_\alpha)$ is the inverse of
the Riordan array of $(\gamma,\alpha)$, or equivalently, that the
umbral composition of the Sheffer sequence of $(\gamma,\alpha)$ with
the Sheffer sequence of $(-1\punt\K_{\gamma,\alpha},-1\punt\K_\alpha)$
is $(x^n)_{n\geq 0}$.
\begin{rem}
Classically, the (ordinary) Riordan array $(a_{n,k})_{n,k}$ of $(f(z),g(z))$ is defined by
\[a_{n,k}=[z^n]f(z)\left(zg(z)\right)^k,\]
with $f(0)\neq 0\neq g(0)$. Without lose of generality, let $f(0)=1=g(0)$ and assume $e^{\gamma z}\simeq f(z)$ and $e^{\alpha z}\simeq g(z)$. We have
\[a_{n,k}\simeq\frac{(\gamma+k\punt\alpha)^{n-k}}{(n-k)!}.\]
If $(s_{n,k})_{n,k}$ is defined by \eqref{id:Rio} then it is plain that the map $(a_{n,k})_{n,k}\mapsto(s_{n,k})_{n,k}$ acts linearly with respect to the matrix multiplication. Indeed, we have
\[\binom{n}{k}(\sigma+k\punt\rho)^{n-k}\simeq\sum_{i=0}^n\binom{n}{i}(\gamma+i\punt\alpha)^{n-i}\binom{i}{k}(\eta+k\punt\delta)^{i-k}\]
if and only if
\[\frac{(\sigma+k\punt\rho)^{n-k}}{(n-k)!}\simeq\sum_{i=0}^n\frac{(\gamma+i\punt\alpha)^{n-i}}{(n-i)!}\frac{(\eta+k\punt\delta)^{i-k}}{(i-k)!}.\]
Finally, since the exponential Riordan group is isomorphic to the group of Sheffer sequences (with umbral composition as operation), then such an isomorphisms extends immediately to ordinary Riordan arrays. This way we recover one of the main results of~\cite{HHS}.
\end{rem}

Now, let $x,y$ be commuting variables, and let $q,t$ be two
parameters. Set $\bar{u}=-1\punt\chi\punt -1$ and consider the
following umbra,
\begin{equation}\label{id:chiumbra}
t\punt\bar{u}+y\punt\chi\punt x\punt\beta\punt(q\punt\bar{u})_\der.
\end{equation}
If $(P_n(x,y;q,t))_{n\geq 0}$ is its sequence of moments, then
from \eqref{id:sheffumbra} and \eqref{id:sheffmom}, we obtain
\[P_n(x,y;q,t)\simeq \sum_{k=0}^n\binom{n}{k}\left(t\punt\bar{u}+kq\punt \bar{u}\right)^{n-k}(y\punt\chi\punt x)^k.\]
Moreover, since
\[\frac{\left(t\punt\bar{u}+kq\punt \bar{u}\right)^{n-k}}{(n-k)!}\simeq \frac{\left((t+kq)\punt \bar{u}\right)^{n-k}}{(n-k)!}\simeq \binom{n-k+t+kq-1}{n-k},\]
and being
\[\frac{(y\punt\chi\punt x)^k}{k!}\simeq  \binom{y}{k}x^k,\]
then the following explicit formula for $P_n(x,y;q,t)$ can be
derived,
\begin{equation}\label{id:formula}P_n(x,y;q,t)=n!\sum_{k=0}^n\binom{n-k+t+kq-1}{n-k}\binom{y}{k}x^k.\end{equation}
Thanks to \eqref{id:chiumbra}, the generating
function of $(P_n(x,y;q,t))_{n\geq 0}$ is
\[1+\sum_{n\geq 1}P_n(x,y;q,t)\frac{z^n}{n!}=\frac{1}{(1-z)^{t}}\left(1+\frac{xz}{(1-z)^q}\right)^y.\]
The polynomial sequence $(P_n(x,y;q,t))_{n\geq 0}$ specializes in
well-known families of polynomials, possibly not Sheffer
sequences, when the variables $x,y$, and the parameters $q,t$ are
chosen in a suitable way. This way, from \eqref{id:formula} we get
easily their expansion in terms of some classical basis.
\begin{ex}[Tchebychev polynomials of the II kind]
These are orthogonal polynomials $(U_n(x))_{n\geq 0}$ with generating function
\[\sum_{n\geq 1}U_n(x)z^n=\frac{1}{1-2xz+z^2}.\]
They are obtained by setting $(x,y;q,t)=(-2x+2,-1;2,2)$. identity
\eqref{id:formula} gives the expansion of $T_n(x)$ in terms of the
basis $((x-1)^n)_{n\geq 0}$,
\begin{align*}
U_n(x)&=\sum_{k=0}^n\binom{n+k+1}{n-k}2^k(x-1)^k.
\end{align*}
\end{ex}
\begin{ex}[Gegenbauer polynomials]
These are orthogonal polynomials $(P_n^{(\lambda)}(x))_{n\geq 0}$ with generating function
\[\sum_{n\geq 1}P_n^{(\lambda)}(t)x^n=\left(\frac{1}{1-2xz+z^2}\right)^\lambda.\]
Of course, they generalize the Tchebychev polynomials of the II kind and are obtained by setting $(x,y;q,t)=(-2x-2,-\lambda;2,2\lambda)$. identity \eqref{id:formula} gives
\[P_n^{(\lambda)}(x)=\sum_{k=0}^n\binom{n+2\lambda+k-1}{n-k}\binom{\lambda+k-1}{k}2^k(x-1)^k.\]
\end{ex}
\begin{ex}[Meixner of I kind]
This orthogonal Sheffer sequence has generating function
\[\sum_{n\geq 0}{\frak m}_n\left(x;b,c\right)\frac{z^n}{n!}=\left(\frac{1}{1-z}\right)^b\left(\frac{1-\frac{1}{c}z}{1-z}\right)^x,\]
with $c\neq 0,1$ and $b\neq 0,-1,-2,\ldots$. So, we recover them when $(x,y;q,t)=((c-1)/c,x;1,b)$ and obtain
\[{\frak m}_n\left(x;b,c\right)=n!\sum_{k=0}^{n}\binom{n+b-1}{n-k}\left(\frac{1-c}{c}\right)^k\binom{x}{k}.\]
\end{ex}
\begin{ex}[Mittlag-Leffler and Pidduck polynomials]
The Mittlag-Leffler polynomials $(M_n(x))_{n\geq 0}$ are the following binomial sequence
\[\sum_{n\geq 0}M_n(x)\frac{z^n}{n!}=\left(\frac{1+z}{1-z}\right)^x.\]
Hence, they can be thought as $M_n(x)={\frak m}_n(x;0,-1)$ and we get them by setting $(x,y;q,t)=(2,x;1,0)$. Their explicit formula in terms of the binomial basis $\left(\binom{x}{n}\right)_{n\geq 0}$ is
\[M_n(x)=n!\sum_{k=1}^{n}\binom{n-1}{n-k}2^k\binom{x}{k}.\]
Finally, Pidduck polynomials arise as the Sheffer sequence given by
\[\sum_{n\geq 0}P_n(x)\frac{z^n}{n!}=\frac{1}{1-z}\left(\frac{1+z}{1-z}\right)^x.\]
We set $(x,y;q,t)=(2,x;1,1)$ and obtain
\[P_n(x)=n!\sum_{k=0}^{n}\binom{n}{n-k}2^k\binom{x}{k}.\]
\end{ex}
\section*{Acknowledgements}
The author thanks Domenico Senato for his comments and suggestions improving the technical quality of this paper.
\end{document}